\documentclass[12pt]{article}
\usepackage[paperwidth=8.5in, paperheight=11in, top=1in, bottom=1in, left=.5in, right=.5in]{geometry}

\usepackage{amsmath,amsfonts,amssymb,epsfig,graphicx}
\usepackage{xspace,multicol,caption}
\usepackage{theorem}
\usepackage{color}
\usepackage{authblk}

\newtheorem{theorem}{Theorem}[section]

\newtheorem{lemma}[theorem]{Lemma}
\newtheorem{question}{Question}

\newcommand{\Z}{\mathbb{Z}}

\newenvironment{proof}[1][Proof]{\begin{trivlist}
\item[\hskip \labelsep {\bfseries #1}]}{\end{trivlist}}

\newcommand{\qed}{\hfill \ensuremath{\Box}}

\begin{document}

\title{\textbf{Answer to an Isomorphism \\Problem in $\mathbb{Z}^2$}}

\author{\textbf{Matt Noble}\\
Department of Mathematics and Statistics\\
Middle Georgia State University\\
matthew.noble@mga.edu}
\date{}
\maketitle

\begin{abstract} 

For $S \subset \mathbb{R}^n$ and $d > 0$, denote by $G(S, d)$ the graph with vertex set $S$ with any two vertices being adjacent if and only if they are at a Euclidean distance $d$ apart.  Deem such a graph to be ``non-trivial" if $d$ is actually realized as a distance between points of $S$.  In a 2015 article, the author asked if there exist distinct $d_1, d_2$ such that the non-trivial graphs $G(\mathbb{Z}^2, d_1)$ and $G(\mathbb{Z}^2, d_2)$ are isomorphic.  In our current work, we offer a straightforward geometric construction to show that a negative answer holds for this question.\\[5pt]    

\end{abstract}

\section{Introduction}

\thispagestyle{empty}

In the \textit{Geombinatorics} tradition, this work will feature a blend of graph theory, geometry, and classical number theory.  We will follow the terminology and notation used in most graph theory texts, and for clarification, one could consult \cite{chartrand}.  In our discussion, we will also employ a number of concepts from elementary number theory -- divisibility, the Chinese remainder theorem, representations of integers as sums of squares.  For a refresher, we recommend \cite{leveque}.

For arbitrary graphs $G_1$ and $G_2$, define $G_1$ to be \textit{isomorphic} to $G_2$, and write $G_1 \simeq G_2$, if and only if there exists a bijection $\varphi: V(G_1) \rightarrow V(G_2)$ such that for any $u, v \in V(G_1)$, $u,v$ are adjacent if and only if $\varphi(u), \varphi(v)$ are adjacent.  Let $f$ be any graph parameter.  If it is guaranteed for isomorphic graphs $G_1$ and $G_2$ that $f(G_1) = f(G_2)$, then $f$ is said to be a \textit{graph invariant}.  A common graph invariant that will play a role in our work is $k(G)$, the number of components of $G$.

The notion of the \textit{Euclidean distance graph} has been a central topic of investigation in \textit{Geombinatorics} articles, both past and present, and that will be the case here as well.  Let $S \subset \mathbb{R}^n$ and $d > 0$.  Define $G(S, d)$ to be the graph whose vertices are the points of $S$, with any two vertices being adjacent if and only if they are a Euclidean distance $d$ apart.  Such a graph is deemed \textit{non-trivial} if $d$ is actually realized as a distance between points of $S$, as otherwise, $G(S, d)$ has an empty edge-set, and is not of interest.  Note that $G(S, d_1) \simeq G(S, d_2)$ if and only if there exists an automorphism $\varphi$ of $S$ such that for any points $u,v \in S$, $|u - v| = d_1$ if and only if $|\varphi(u) - \varphi(v)| = d_2$.  The \textit{isomorphism classes} of distance graphs on $S$ are formed by partitioning the interval $(0, \infty)$ such that for any set $P$ of the partition, and $d_1, d_2 \in P$, $G(S, d_1) \simeq G(S, d_2)$, and $P$ is maximal with respect to this property.  As an easy example, note that the space $\mathbb{R}^n$ has only one isomorphism class, as any graph $G(\mathbb{R}^n, d)$ with $d > 0$ is isomorphic to the unit-distance graph $G(\mathbb{R}^n, 1)$ by an obvious scaling argument.  Note also that with this definition, we are playing a little fast and loose with the terminology, as one would normally expect an isomorphism class of graphs to actually consist of a collection of graphs, not a collection of distances.  However, considering the natural correspondence of a graph $G(S, d)$ to the distance $d$, there should be no confusion or sacrifice of mathematical precision in this setup.

As a quick warmup problem, we observe that for distinct $d_1,d_2$, the non-trivial graphs $G(\mathbb{Z}, d_1)$ and $G(\mathbb{Z}, d_2)$ are not isomorphic.  Consider the set of all distance graphs with vertex set $\mathbb{Z}$.  A distance $d$ is realized between distinct points of $\mathbb{Z}$ if and only if $d$ is a positive integer, so the trivial isomorphism class (that is, those distances that produce empty graphs) consists of all $d \in \mathbb{R}^+ \setminus \mathbb{Z}$.  For a positive integer $d$, note that $G(\mathbb{Z}, d)$ consists of $d$ components, with each of those components being isomorphic to $G(\mathbb{Z}, 1)$.  For distinct $d_1, d_2 \in \mathbb{Z}^+$, we therefore have $k(\mathbb{Z}, d_1) \neq k(\mathbb{Z}, d_2)$ and thus $G(\mathbb{Z}, d_1) \not \simeq G(\mathbb{Z}, d_2)$.                      

In a 2015 article \cite{isomorphism}, the author studied isomorphism classes of distance graphs with vertex set $\mathbb{Z}^2$ and posed the following question.

\begin{question} \label{z2question}
For distinct $d_1, d_2 > 0$, both realized as distances between points of $\Z^2$, is it possible that the graphs $G(\Z^2, d_1)$ and $G(\Z^2, d_2)$ are isomorphic?
\end{question}

\noindent Partial results were obtained in \cite{isomorphism}, and we will present them in the next section as a jumping off point for our present work.  These partial conclusions seemed to point to Question \ref{z2question} having a negative answer, but a full resolution was not achieved. 

In Section 2, we again take up the mantle on this problem.  We will offer a relatively straightforward construction showing that indeed, the question does have a negative answer, or, in other words, every non-trivial isomorphism class of Euclidean distance graphs with vertex set $\Z^2$ consists of a single graph.  In Section 3, we conclude with a few thoughts concerning directions for further investigation.  

\section{An Isomorphism Problem in $\mathbb{Z}^2$}

In this section, let $\mathcal{P} \subset \Z^+$ be the set of all primes congruent to 1 modulo 4.  Also throughout this section, we will assume that any $G(\Z^2, d)$ is non-trivial.  In other words, $d = \sqrt{r}$ for some $r \in \Z^+$ with $r$ being representable as a sum of two integer squares.  These $r$ are given by a well-known theorem of Euler, which we give for quick reference as Lemma \ref{twosquareslemma}.

\begin{lemma} \label{twosquareslemma} A positive integer $r$ may be written as $r = a^2 + b^2$ for $a, b \in \mathbb{Z}$ if and only if in the prime factorization of $r$, prime factors congruent to 3 modulo 4 each appear to an even degree.
\end{lemma}

Given some $r$ which is representable as a sum of two integer squares, a characterization of the possible solutions $a,b$ to $r = a^2 + b^2$ is given in many introductory number theory texts (for example, see Chapter 7 of \cite{leveque}).  

\begin{lemma} \label{twosquaresrepslemma} Let $r \in \Z^+$ with prime factorization\\ $r = 2^\gamma {p_1}^{\alpha_1}\cdots {p_m}^{\alpha_m} {q_1}^{2\beta_1} \cdots {q_n}^{2\beta_n}$ where $p_1, \ldots, p_m, q_1, \ldots q_n$ are distinct primes with each $p_i \equiv 1 \pmod 4$ and each $q_j \equiv 3 \pmod 4$.  The following are both true.
\begin{enumerate}
\item[(i)] If $\gamma \leq 1$ and $\beta_1 = \cdots = \beta_n = 0$, then there exist $a,b \in \Z$ such that $r = a^2 + b^2$ and $\gcd(a,b) = 1$.  
\item[(ii)] Let $r = c^2 + d^2$.  Let $g = \gcd(c,d)$ and $h = 2^{\lfloor\frac{\gamma}{2}\rfloor}{q_1}^{\beta_1} \ldots {q_n}^{\beta_n}$.  Then $h | g$.  
\end{enumerate}
\end{lemma}  

In \cite{isomorphism}, a version of Lemma \ref{twosquaresrepslemma} was used to determine $k(\Z^2, d)$ for all $d$.  Moreover, it was shown that if $r_1$ is such that all of its prime factors are elements of $\mathcal{P}$, and $r_2 = r_1(2)^\gamma {q_1}^{2\beta_1} \cdots {q_n}^{2\beta_n}$ with each $q_j \equiv 3 \pmod 4$, then the graph $G(\Z^2, \sqrt{r_2})$ has each of its components isomorphic to the entire graph $G(\Z^2, \sqrt{r_1})$.  This observation was utilized to show that all one needs to do to establish that Question \ref{z2question} has a negative answer is to show that for distinct $r_1, r_2 \in \Z^+$ with $r_1, r_2$ each having prime factorizations consisting solely of elements of $\mathcal{P}$, the corresponding $G_1 = G(\Z^2, \sqrt{r_1})$ and $G_2 = G(\Z^2, \sqrt{r_2})$ are not isomorphic.  Furthermore, this was done successfully in \cite{isomorphism} in the case of $r_1, r_2$ both prime, through analysis of a particular graph invariant, the number of closed walks of a specified length containing a vertex $v \in V(G_1)$ and its image $\varphi(v) \in V(G_2)$ under an assumed isomorphism $\varphi: G_1 \rightarrow G_2$.  Unfortunately, this line of proof appeared very difficult to navigate when $r_1$ or $r_2$ is composite, as in that situation, the enumeration of closed walks does not appear feasible.

We now consider a different graph invariant.  Let $G_1, G_2$ be arbitrary graphs, and suppose $\varphi: G_1 \rightarrow G_2$ is an isomorphism.  Let $u,v \in V(G_1)$, and for positive integer $l$, define $f_l(u,v)$ as the number of distinct paths of length $l$ beginning at $u$ and terminating at $v$.  Clearly, $f_l(u,v) = f_l(\varphi(u),\varphi(v))$.  Now denoting $G_1 = G(\Z^2, \sqrt{r_1})$ and $G_2 = G(\Z^2, \sqrt{r_2})$ with $r_1, r_2$ each having prime factorizations consisting solely of elements of $\mathcal{P}$, our plan will be to assume the existence of an isomorphism $\varphi: G_1 \rightarrow G_2$, and in the following lemmas, develop a few conditions to which $\varphi$ must adhere.  We then combine these observations in Theorem \ref{z2result} to obtain a contradiction, thus showing Question \ref{z2question} has a negative answer.  Throughout, we may without loss of generality freely assume that $\varphi$ fixes the origin and that $r_1 > r_2$.

\begin{lemma} \label{collinearlemma} Let $C_1, C_2$ be circles centered at the origin and having radii $\sqrt{r_1}, \sqrt{r_2}$, respectively.  Let point $p \in C_1 \cap \Z^2$.  There exists a point $q \in C_1 \cap \Z^2$ such that for all $n \in \Z^+$, $\varphi(np) = nq$.  
\end{lemma}

\begin{proof} Designate $u = (0,0)$.  Let $n \in \Z^+$, and note that $f_n(u, np) = 1$.  Denote by $C$ the circle centered at $u$ and having radius $n\sqrt{r_1}$, and note that for any point $\alpha \in \Z^2$ with $\alpha \not \in C$, $f_n(u, \alpha) \neq 1$.  This is straightforward to see, as if $\alpha$ lies outside of $C$, then $f_n(u, \alpha) = 0$.  If $\alpha$ lies inside $C$, and there exist $\alpha = v_0, v_1, \ldots, v_n = u$ constituting the vertices of a path of length $n$ in $G_1$, then there exists some $i \in \{0, \ldots, n-2\}$ with $|v_i - v_{i+2}| < 2\sqrt{r_1}$, which implies the existence of a lattice point $w \neq v_{i+1}$ such that $v_0, \ldots, v_i, w, v_{i+2}, \ldots, v_n$ also form a path of length $n$ in $G_1$.  

We now proceed by way of induction.  Note that the lemma holds in the case of $n = 1$ simply by the definition of $\varphi$ being an isomorphism.  Assume it holds for all $n \leq k - 1$.  In $G_1$, we have $f_k(u,kp) = 1$, and so in $G_2$, $f_k(u,\varphi(kp)) = 1$ as well.  Thus $\varphi(kp)$ lies on a circle $C'$ centered at $u$ and having radius $k\sqrt{r_2}$.  However, $\varphi(kp)$ must also be at distance $\sqrt{r_2}$ from $(k-1)q$.  A circle $C''$ of radius $\sqrt{r_2}$ and centered at $(k-1)q$ intersects $C'$ at exactly the point $kq$.  This completes the induction step and establishes proof of the lemma.\qed
\end{proof}

\begin{lemma} \label{anglelemma} There exist vectors $v_\alpha, v_\beta \in \Z^2$, each of length $\sqrt{r_1}$, such that the angle $\theta$ between $v_\alpha, v_\beta$ is not realized between any pair of vectors $w_\alpha, w_\beta \in \Z^2$, each of length $\sqrt{r_2}$.
\end{lemma}

\begin{proof} Since by assumption $r_1 > r_2$, for some prime $p \in \mathcal{P}$ and positive integer $n$, it is the case that $p^n$ divides $r_1$ but does not divide $r_2$.  By Lemma \ref{twosquaresrepslemma}, there exist $a,b \in \Z^+$ such that $a^2 + b^2 = r_1$ and $\gcd(a,b) = 1$.  Note here that $p$ divides neither $a$ nor $b$.  

Let $v_\alpha = \langle a,b \rangle$ and $v_\beta = \langle b,a \rangle$.  The angle $\theta$ is then given by $\cos(\theta) = \frac{v_\alpha \cdot v_\beta}{|v_\alpha||v_\beta|} = \frac{2ab}{r_1}$.  Similarly, letting $w_\alpha, w_\beta \in \Z^2$ be vectors of length $\sqrt{r_2}$ and $\theta'$ the angle between them, we have $\cos(\theta') = \frac{w_\alpha \cdot w_\beta}{r_2}$.  Setting $\cos^{-1}(\frac{2ab}{r_1}) = \cos^{-1}(\frac{w_\alpha \cdot w_\beta}{r_2})$ gives a contradiction as it implies $2abr_2 = r_1(w_\alpha \cdot w_\beta)$ where $p^n$ divides $r_1(w_\alpha \cdot w_\beta)$ but $p^n$ does not divide $2abr_2$.\qed
\end{proof}

\begin{lemma} \label{10lemma} Let $r \in \Z^+$ where any prime factor of $r$ is in $\mathcal{P}$.  Let $u,v \in \Z^2$ where $|u - v| < \sqrt{r}$.  There exists a path in $G(\Z^2, \sqrt{r})$ beginning at $u$ and terminating at $v$ and having length $n < 8r^\frac32$.
\end{lemma}

\begin{proof} By Lemma \ref{twosquaresrepslemma}, let $a,b \in \Z^+$ where $a^2 + b^2 = r$ and $\gcd(a,b) = 1$.  Since $r$ is odd, one of $a,b$ is odd and the other even, so without loss of generality, assume $a$ even.  By the Chinese remainder theorem, there exist $s, t \in \Z^+$ such that $sa - tb = -1$.  Note here that $s+t < a+b$.  As observed in Theorem 3.2 of \cite{isomorphism}, the following sum of vectors yields $\langle 0, 1 \rangle$. 

\begin{center}
$\langle a, b \rangle + \frac{a}{2}\left[s \langle a, b \rangle + s \langle a, -b \rangle + t \langle -b, a \rangle + t \langle -b, -a \rangle \right] + \frac{b - 1}{2} \left[s \langle b, a \rangle + s \langle -b, a \rangle + t \langle a, -b \rangle + t \langle -a, -b \rangle \right]$   
\end{center}

This sum adds together $(s+t)(a+b-1) + 1$ vectors in total, and by negating or permuting entries of those vectors used, we may also form any of $\langle 1, 0 \rangle$, $\langle -1, 0 \rangle$, or $\langle 0, -1 \rangle$ by summing a similar number of vectors, each of length $\sqrt{r}$.  Let $\langle x,y \rangle$ be the vector with initial point $u$ and terminal point $v$.  We have each of $|x|,|y|$ less than $\sqrt{r}$ (as well as each of $a,b$ less than $\sqrt{r}$), so the desired path can be found in $G(\Z^2, \sqrt{r})$ of length $(s+t)(a+b-1)(|x| + |y|)$ which is in turn less than $8r^\frac32$.\qed
\end{proof}

We remark in passing that the order of the ``$8r^\frac32$" bound in the statement of Lemma \ref{10lemma} can almost certainly be improved.  Perhaps that may even merit investigation in some future research program.  However, the present form of Lemma \ref{10lemma} is fine for our its intended use in the proof of Theorem \ref{z2result}.

\begin{theorem}  \label{z2result} Let distinct $r_1, r_2 \in \Z^+$, each having prime factorizations consisting solely of elements of $\mathcal{P}$.  Then $G(\Z^2, \sqrt{r_1}) \not \simeq G(\Z^2, \sqrt{r_2})$.
\end{theorem}

\begin{proof} As in previous discussion, let $G_1 = G(\Z^2, \sqrt{r_1})$ and \\$G_2 = G(\Z^2, \sqrt{r_2})$ with $r_1 > r_2$, and assume the existence of an isomorphism $\varphi: G_1 \rightarrow G_2$ that fixes the origin.  Regarding Lemma \ref{anglelemma}, let $v_\alpha, v_\beta \in \Z^2$ be vectors of length $\sqrt{r_1}$ where the angle $\theta_1$ between $v_\alpha, v_\beta$ is not realized between any pair of vectors in $\Z^2$ having length $\sqrt{r_2}$.

Consider the two sets of collinear points $\{v_\alpha, 2v_\alpha, 3v_\alpha, \ldots\}$ and\\ $\{v_\beta, 2v_\beta, 3v_\beta, \ldots\}$.  By Lemma \ref{collinearlemma}, for some vectors $w_\alpha, w_\beta$, each of length $\sqrt{r_2}$, $\varphi(iv_\alpha) = iw_\alpha$ and $\varphi(iv_\beta) = iw_\beta$ for all $i \in \Z^+$.  Let $\theta_2$ be the angle between vectors $w_\alpha, w_\beta$.  We have $\theta_1 \neq \theta_2$, and we will assume $\theta_1 < \theta_2$, however a similar argument would hold in the case of $\theta_2 < \theta_1$.  Note that, since $\theta_1 < \theta_2$, by taking larger and larger $m,n$, we can make the difference $|mw_\alpha - nw_\beta| - |mv_\alpha - nv_\beta|$ arbitrarily large as well. 

Let $\ell$ be the ray containing each of $(0,0), v_\beta, 2v_\beta, 3v_\beta, \ldots$, and for each $i \in \Z^+$, let $C_i$ be the circle centered at $iv_\beta$ and having radius $\sqrt{r_1}$.  Form support lines $\ell_1, \ell_2$ by translating $\ell$ by vectors $t, -t$ where $t$ is perpendicular to $\ell$ and has length $\frac{\sqrt{3r_1}}{2}$.  For a visual reference, see Figure 1.  Let $v \in \Z^2$ be a vector of length $\sqrt{r_1}$ where $v \not \in \{v_\alpha, v_\beta\}$, so that we may start at $mv_\alpha$ for some sufficiently large $m$ and form points $mv_\alpha + v, mv_\alpha + 2v, mv_\alpha + 3v, \ldots$ such that the ray containing each of $mv_\alpha, mv_\alpha + v, mv_\alpha + 2v, mv_\alpha + 3v, \ldots$ intersects $\ell$.  Such a vector $v$ is guaranteed to exist, and in fact, if one uses $v_\alpha = \langle a,b \rangle, v_\beta = \langle b,a \rangle$ with $a^2 + b^2 = r_1$ and $a,b \in \Z^+$ (as described in the proof of Lemma \ref{anglelemma}), $v = \langle a, -b \rangle$ is suitable.

\begin{figure}
    \centering
    \includegraphics[scale=.8]{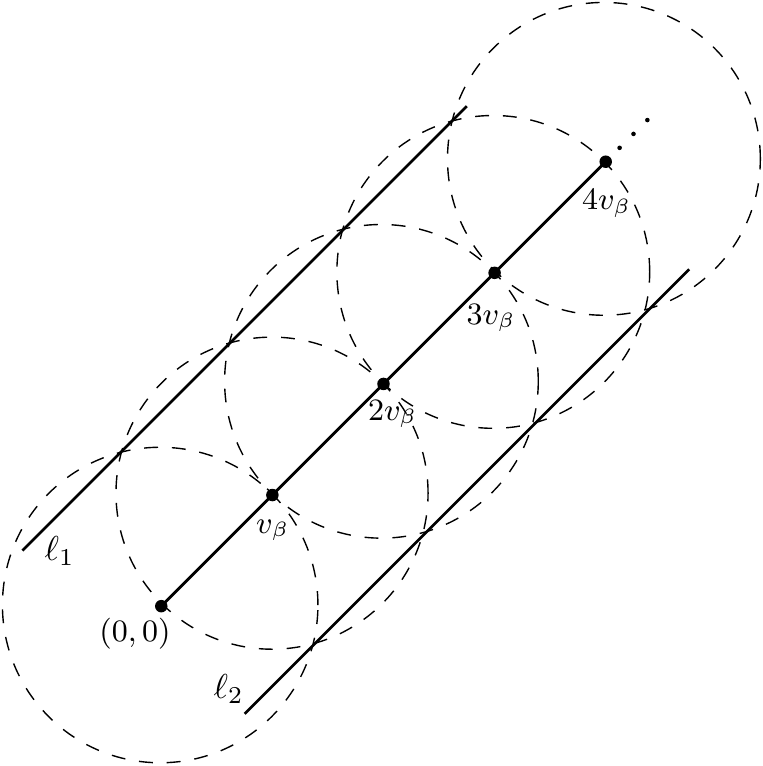}
    \caption{}
    \label{circlepicture}
\end{figure}

Note that there must be some $k \in \Z^+$ such that the point $p_k = mv_\alpha + kv$ falls between $\ell_1, \ell_2$, and in doing so, we have for some $n$, $|p_k - nv_\beta| < \sqrt{r_1}$.  By Lemma \ref{10lemma}, there exists a path in $G_1$, beginning at $p_k$ and terminating at $nv_\beta$, of length less than $8r^\frac32$.  Letting $z = k + 8r^\frac32$, in $G_1$ we therefore have $f_z(mv_\alpha, nv_\beta) > 0$.  However, for $m$ taken sufficiently large, $f_z(mw_\alpha, nw_\beta) = 0$, a contradiction that completes the proof of the theorem.\qed
\end{proof}  

Theorem \ref{z2result}, along with the previously described work of \cite{isomorphism}, give us the following main result.

\begin{theorem} \label{mainresult} Let distinct $d_1, d_2$ be distances realized between points of the integer lattice $\Z^2$.  Then the graphs $G(\Z^2, d_1)$ and $G(\Z^2, d_2)$ are not isomorphic. 
\end{theorem}

\section{Further Work}

Along with Question \ref{z2question}, the following was posed in \cite{isomorphism}.

\begin{question} \label{subgraphquestion} Let $G_1 = G(\Z^2, d_1)$ and $G_2 = G(\Z^2, d_2)$ where $G_1 \not \simeq G_2$.  If possible, construct a finite graph $H$ which appears as a subgraph of exactly one of $G_1, G_2$.
\end{question}
   
Of course, the existence of such $H$ would guarantee $G_1 \not \simeq G_2$ without the a priori knowledge of Theorem \ref{mainresult} that the two graphs are not isomorphic.  Note also the the qualifier ``If possible" is necessary, as for some instances of $G_1, G_2$, no such $H$ exists.  As touched upon in the previous section, consider $G_1 = G(\Z^2, 1)$ and $G_2 = G(\Z^2, \sqrt{2})$.  Graph $G_1$ is connected, while $G_2$ is a disconnected graph having two components, each of which is isomorphic to $G_1$.  Even without access to Theorem \ref{mainresult}, we then have that $G_1 \not \simeq G_2$, but here, the construction of the desired $H$ is not possible.  

Unfortunately, the line of proof given in the previous section, when taken at face value, does not offer much toward a resolution of Question \ref{subgraphquestion}.  Really, the only graphs that are explicitly noted in the proof as being subgraphs of $G(\Z^2, \sqrt{r_1})$ are paths and even cycles.  Certainly, these are subgraphs of any $G(\Z^2, \sqrt{r_2})$ as well.  The characteristics of $G(\Z^2, \sqrt{r_1})$ that were utilized in the proof of Theorem \ref{mainresult} were more geometric in nature, as opposed to being graph or number theoretic.  However, perhaps in some future investigation, the method used in Section 2 can be modified to elaborate upon the underlying structural properties of $G(\Z^2, d)$ in a more concrete manner, and hopefully shed some light on Question \ref{subgraphquestion}.


\end{document}